\documentclass[12pt]{article}
\usepackage{amsmath,amsfonts,amssymb,amsthm,amscd}
\usepackage{xcolor}
\title{Notes on  Lie derivatives, algebraic $D$-varieties, and Ax's theorem}
\date{\today}
\author{Anand Pillay\thanks{Supported by NSF grants  DMS-2054271 and DMS-2502292}\\{University of Notre Dame}}

\newtheorem{Theorem}{Theorem}[section]
\newtheorem{Proposition}[Theorem]{Proposition}
\newtheorem{Definition}[Theorem]{Definition}
\newtheorem{Remark}[Theorem]{Remark}
\newtheorem{Lemma}[Theorem]{Lemma}
\newtheorem{Corollary}[Theorem]{Corollary}
\newtheorem{Fact}[Theorem]{Fact}

\newcommand{\Q}{\mathbb Q}  
\newcommand{\Z}{\mathbb Z}

\newcommand{\C}{\mathbb C}

\begin{document}
\maketitle

\begin{abstract} 
We discuss the relationship between   Lie derivatives and the linear differential equations on cotangent spaces of algebraic $D$-varieties at sharp points (from  \cite{Pillay-Ziegler}). We also  take the  liberty to give an account of Ax's theorem  (Ax-Schanuel for the exponential map)

\end{abstract}

\section{Introduction}
This note is partially motivated by discussions in a model theory meeting in Freiburg (November 2023) around Lie derivatives and their role in the proof of Ax's theorem and related results, following a talk by R\'emi Jaoui and comments by Piotr Kowalski.  The  key object is a certain ``linear differential equation" on a certain space of K\"ahler differentials. On the other hand, in \cite{Pillay-Ziegler}  we introduced a certain linear differential equation on the cotangent space and higher cotangent spaces of an algebraic $D$-variety at a so-called $\sharp$-point as  a tool for giving somewhat direct accounts of function field Mordell-Lang in characteristic $0$ as well as the Zilber dichotomy for strongly minimal sets in differentially closed fields.  These seemed to be closely related.  In this note, given a rational  algebraic $D$-variety $(X,\partial ')$ over a differential field $K$, we point out how to deduce the Lie derivative on the  space $\Omega_{K}(X)$ and its properties, from these linear differential equations on the cotangent spaces to $X$ at ``sharp"-points.   (Also extended to proalgebraic $D$-varieties.) This is the only novelty of the note (if there is any novelty).  See Corollary 2.3.

Another motivation comes from a topics course on the model theory of differential fields, taught by me at Notre Dame in Autumn 2025.  I had planned to finish the course with a proof of Ax's theorem (for integer powers of the multiplicative group).  However I only understood what was going on (and the relation to \cite{Pillay-Ziegler}) after the course was finished.  So this note represents in  a sense the completion of the course.  Also maybe this will fill a gap in the literature in the sense that  I  could not find a clear conceptual account of Ax's original theorem from \cite{Ax}. (Although Dave Marker has seminar notes on Ax's theorem \cite{Marker}  with an efficient proof making use of ideas of Rosenlicht.) 
So this note may also help as an entry point for students, for example, to Ax's theorem.

Let us mention in passing that algebraic $D$-varieties $(X,D)$ correspond (by passing to the set  $(X,D)^{\sharp}$ of $\sharp$-points) to finite dimensional differential algebraic varieties, and the solution set of the linear differential on the cotangent space at a sharp point of $X$, is precisely the Kolchin cotangent space of the differential algebraic variety $(X,D)^{\sharp}$. at the relevant point.  An analogous thing holds for arbitrary (not necessarily finite-dimensional) differential algebraic varieties.

Our algebraic geometric notation is standard (such as in \cite{Shafarevich}).  For the relevant model theory of differential fields see \cite{Pillay-Ziegler}. 

Many thanks to Piotr Kowalski for his suggestions and comments. 

\subsection{Rational differential forms.}
Let us begin with recalling `the ``classical" notion of (rational) differential $1$-form.
We work in characteristic $0$, and take $K$ to be an algebraically closed field. Let $X$ be an irreducible algebraic variety over $K$, affine if one wishes ($X\subseteq K^{n}$) , and let $K(X)$ be the field of rational functions on $K$.
If $\alpha$ is a generic point of $X$ over $K$ then $K(X)$  coincides  with $K(\alpha)$   (for example by the map taking the coordinate variables $x = (x_{1},..,x_{n})$ on $X$ to $\alpha$).

Recall that for $\alpha\in X(K)$ the tangent space to $X$ at $\alpha$ is the vector subspace of $K^{n}$ defined by $\sum_{i=1,..,n} (\partial P/\partial x_{i})(\alpha) u_{i} = 0$ as $P$ ranges over a finite set of generators of $I_{K}(X)$. 
The cotangent space to $X$ at $\alpha$ is the dual space $T(X)_{\alpha}^{*}$ of linear forms on $T(X)_{\alpha}$.

Given $\alpha\in X(K)$, the local ring ${\mathcal O}_{K}(X)_{\alpha}$ of $X$ at $\alpha$ is the set of $f\in K(X)$ which are defined at $\alpha$. For such $f $, $df_{\alpha}$ is the linear form 
$\sum_{i=1,..,n}(\partial f/\partial x_{i})(\alpha)u_{i}$.  Note that if $f(\alpha) = c\in K$, then $df_{\alpha} = d(f-c)_{\alpha}$, whereby we may restrict to those $f$ in  the maximal ideal $ {\mathcal M}_{\alpha}$ consisting of  elements of the local ring at $\alpha$ which are $0$ at $\alpha$.

\begin{Fact} $f\to df_{\alpha}$ defines an isomorphism of $K$-vector spaces between ${\mathcal M}_{\alpha} /{\mathcal M}_{\alpha}^{2}$ and $T(X)_{\alpha}^{*}$ where ${\mathcal M}_{\alpha}^{2}$ is the ideal in the local ring at $\alpha$ generated by all products $fg$ of $f,g\in {\mathcal M}_{\alpha}$.
\end{Fact} 

Note that if $f_{1},..,f_{r}\in {\mathcal O}(X)_{\alpha}$ and $c_{1},..,c_{r}\in K$ then $\sum_{i=1,..,r}c_{i}df_{i} = d(\sum_{i=1,..,r}c_{i}f_{i})$. 

We denote by $\Omega(X)$ or $\Omega_{K}(X)$ the $K(X)$-vector space of rational differential forms (over $K$ on $X$).  Recall that this means that  ${\Omega}(X)$ is the collection of objects $\omega = \sum_{i=1,..,m}f_{i}dg_{i}$ where $f_{i}, g_{i}\in K(X)$. 
For any $\alpha\in X(K)$ at which each $f_{i}$ and $g_{i}$ are defined,  $\omega_{\alpha} = \sum_{i=1,..,n}  f_{i}(\alpha)(dg_{i})_{\alpha} \in T(X)_{\alpha}^{*}$.  More precisely for $g\in K(X)$ defined at $\alpha$, $dg_{\alpha}$ is the linear form 
on $T(X)_{\alpha}$ defined above, and extend to arbitrary $\omega_{\alpha}$ by linearity. 

We consider two such rational differential differential forms $\omega_{1}$, $\omega_{2}$ on $X$ to be equal if $(\omega_{1})_{\alpha}  =( \omega_{2})_{\alpha}$ on a Zariski open subset of $X(K)$. 
See Section 5.2 of  Chapter III of \cite{Shafarevich} for the following: 
\begin{Fact} $\Omega(X)$ is an $n$-dimensional vector space over $K(X)$, which is moreover defined by the relations $d(f+g) = df + dg$, $d(fg) = fdg + gdf$, and $d(a)= 0$ for $a\in K$ (i.e. for constant functions). 
\end{Fact}

The general notion and theory of K\"ahler differentials was introduced to give a purely algebraic account of differential forms (rational as well as regular). $\Omega_{K}(X)$ as defined above coincides with the $K(X)$-vector space of K\"ahler differentials (over $K$), $\Omega(K(X)/K)$.

We will be discussing proalgebraic varieties. For our purposes a proalgebraic variety over $K$ is an inverse limit  $X_{\infty} = \varprojlim X_{n}$ of (affine irreducible) algebraic varieties $X_{n}$ over $K$ where for each $n$ we have a dominant morphism $\pi_{n}:X_{n+1}\to X_{n}$.  Then we naturally have inclusions $K(X_{n}) \subseteq K(X_{n+1})$ of function fields, and the union is $K(X_{\infty})$. Likewise we have inclusions of the $\Omega_{X_{n}}$ whose union is $\Omega_{K}(X_{\infty})$ a $K(X_{\infty})$-vector space.

\subsection{Rational $D$-varieties.}  
We now assume that the algebraically closed field $K$ is equipped with a derivation $\partial$. So $(K,\partial)$ is an algebraically closed (ordinary) differential field.    We can embed $(K,\partial)$ in a (saturated) differentially closed field, ${\mathcal U}$, and sometimes it is convenient to assume $(K,\partial)$ itself to be differentially closed. 

Algebraic $D$-varieties were essentially introduced by Buium \cite{Buium} as a tool for understanding finite-dimensional sets  and groups definable in differentially closed fields. But see also \cite{Pillay-Ziegler} and \cite{Pillay-Tehran}.

Let again $X$ be an irreducible algebraic variety over $K$ living in affine $n$-space,  as in 1.1.  The ``shifted" tangent bundle of $X$ is a torsor for the tangent bundle of $X$.
\begin{Definition} $T_{\partial}(X)$ is the variety in affine $2n$-space defined  over $K$  by the set of polynomial equations $P({\bar x}) = 0$ defining $X$ together with
the set of equations  $\sum_{i=1,...,n} (\partial P/\partial x_{i})({\bar x})u_{i} +  P^{\partial}({\bar x}) = 0$, as $P$ ranges over a generating set of $I_{K}(X)$ and $P^{\partial}$ is the result of hitting the coefficients of $P$ by $\partial$.
\end{Definition} 

When $K$ is a field of constants, $T_{\partial}(X)$ coincides with $T(X)$ (the tangent bundle of $X$).

\begin{Definition} 
By the structure of a rational $D$-variety on $X$ over $K$ we mean an extension $\partial '$ of  the derivation $\partial$ of $K$  to a derivation of the function field $K(X)$.
\end{Definition}

\begin{Fact} (i) Let $(X,\partial '$) be a rational $D$-variety over $K$.  Let $x_{1},..,x_{n}$ be the coordinate functions on $X$, and let $\partial ' (x_{i}) = s_{i}({\bar x}) \in K(X)$. Then 
$(s_{1},...,s_{n})$ is a rational section of $T_{\partial}(X) \to X$ defined over $K$. 
\newline
(ii) Conversely if ${\bar s}({\bar x}) = (s_{1}({\bar x}),..,s_{n}({\bar x}))$ is a rational section of $T_{\partial}(X) \to X$ defined over $K$ then defining $\partial ' (f) = \sum_{i=1,..,n}(\partial f/\partial x_{i})({\bar x})s_{i}({\bar x}) + f^{\partial}({\bar x})$  for $f\in K(X)$ gives a rational $D$-variety structure on $X$. Namely $\partial '$ is a derivation of $K(X)$ extending $\partial$. 
\end{Fact} 

So we can write a rational $D$ variety over $K$ as either $(X,\partial')$ or $(X,{\bar s})$ where ${\bar s}$ is as in Fact 1.5 (i). 

\begin{Definition}  Let $(X,{\bar s})$ be a rational $D$ variety over $K$. Then by a $\sharp$-point of $X(K)$ we mean some $\alpha = (\alpha_{1},..,\alpha_{n})\in X(K)$ at which ${\bar s}$ is defined and with $\partial (\alpha) = {\bar s}(\alpha)$ (coordinatewise). 
\end{Definition}

The collection of ${\sharp}$-points of $X(K)$ is a quantifier-free definable set in $(K,\partial)$. It may be empty, but:
\begin{Fact}  If $(K,\partial)$ is differentially closed and $(X,{\bar s})$ is a rational $D$-variety over $K$ then the set of $\sharp$-points of $X(K)$ is Zariski-dense in $X(K)$. 
\end{Fact} 
\begin{proof} On the the one hand this is part of the so-called ``geometric axioms" for $DCF_{0}$. On the other hand it is an immediate consequence of the definitions, and model completeness of
 $DCF_{0}$:
So consider the associated derivation $\partial '$ on $K(X)$ extending $\partial$. Let  $\alpha$ be a generic over $K$ point of $X$ in a larger field $L$. So $K(\alpha)$ is isomorphic to $K(X)$ (take $\alpha = (\alpha_{1},..,\alpha_{n})$ to the coordinate functions $x_{1},..,x_{n}$). So  via this isomorphism $\partial '$ gives a derivation of $K(\alpha)$  which extends $\partial$ on $K$ and such that $\partial '(\alpha) = {\bar s}(\alpha)$. 
We can embed $(K(\alpha),\partial ')$ into $({\mathcal U},\partial)$ over $K$, namely we may assume already that $K(\alpha) < {\mathcal U}$ and ${\partial} ' = \partial |K(\alpha)$.
 But $\alpha$ is a generic point of $X$ over $K$, so as $K$ is an elementary substructure of ${\mathcal U}$ for any nonempty Zariski open subset $U$of $X(K)$, there is $\beta \in U$ with $\partial(\beta) = {\bar s}(\beta)$. 

\end{proof} 

Again we can make sense of a rational proalgebraic $D$-variety  $(X_{\infty},\partial ')$ = $(X_{\infty}, {\bar s})$ over a differential field $K$. However this is {\em not} on the face of it simply an inverse limit of rational $D$-varieties. It is 
rather a rational $D$-variety $X_{\infty}$ over $K$ as in Section 1.1, together with lifting of the derivation $\partial$ on $K$ to a derivation $\partial '$ of the function field $K(X_{\infty})$.  The various Facts above hold in this more general 
context, as well as Fact 1.7, although may require $K$ to be somewhat saturated.

\section{The ``Lie derivative"}

First recall that given say a field $K$ with a derivation $\partial$, by a $\partial$-module over $K$ we mean a $K$-vector space $V$ equipped with an additive endomorphism $D_{V}: V\to V$ such that for any $c\in K$ and $v\in V$,
$D_{V}(c\cdot v) = c\cdot D_{V}(v) + \partial (c)\cdot v$. The set of solutions of $D_{V}(v) = 0$ is a $C_{K}$-vector space where $C_{K}$ is the field of constants of $K$. 

 A key point is:
\begin{Fact} Let $v_{1},..,v_{m}\in V$ be solutions of $D_{V} = 0$. Then the $v_{i}$ are $K$-linearly dependent iff they are $C_{K}$-linearly independent. 
\end{Fact} 
\begin{proof}   We may assume that no proper subset of $\{v_{1},..,v_{m}\}$ is $K$-linearly dependent
Let $c_{2},..,c_{m}\in K$ not all $0$ such that $v_{1} + c_{2}v_{2} + ... + c_{m}v_{m} = 0$.  Applying $D_{V}$ and using that $D_{V}(v_{i}) = 0$ for all $i$ we see that
$\partial(c_{2}) = ... = \partial (c_{m}) = 0$, hence the $c_{i}\in C_{K}$. 
\end{proof} 

When $V$ is finite-dimensional and $K$ is differentially closed one can find a $K$-basis of $V$ consisting of solutions of $D_{V} = 0$. 

We  will describe how given a rational $D$-variety $(X,\partial ')$ over a differential field  $(K,\partial)$,  the derivation $\partial '$ of $K(X)$ directly induces a the structure of a $\partial '$-module on the $K(X)$-vector space $\Omega_{K}(X)$ of rational differential forms. 

This can be accomplished in different ways (such as purely algebraically) but it is amusing (and this is the main point) to make use of \cite{Pillay-Ziegler} where we did the same thing, but working instead with the
the cotangent space  $T(X)_{\alpha}^{*}$ of $X$ at a $\sharp$-point $\alpha$ of $X(K)$.  We summarise. 

So we assume that $(X, \partial ') = (X, {\bar s})$ is  a rational $D$-variety over $K$ 
$\alpha = (\alpha_{1},..,\alpha_{n})$ is a $\sharp$-point of $X(K)$ and let $V = T(X)_{\alpha}^{*}$, a $K$-vector space. 
\begin{Lemma} Let $f\in {\mathcal O}(X)_{\alpha}$. 
\newline
(i) Then ${\partial} '(f)\in {\mathcal O}(X)_{\alpha}$, and $d({\partial '}(f))_{\alpha}$ depends only on $df_{\alpha}$.
\newline
(ii) Define $D_{V}:V\to V$ by $D_{V}(df_{\alpha}) = d(\partial ' f)_{\alpha}$. Then $(V,D_{V})$ is a $\partial$-module over $K$. 
\end{Lemma}
\noindent
{\em Explanation.}    (i) As   $\partial '(f) = \sum_{i=1,..,n}(\partial f/\partial x_{i})({\bar x})s_{i}({\bar x}) + f^{\partial}({\bar x})$ and ${\bar s}$ is defined at $\alpha$,  we see that $\partial ' (f) \in {\mathcal O}(X)_{\alpha}$.  As $\alpha$ is a $\sharp$-point of $X(K)$, we see also that 
\newline
(*) $\partial ' (f)(\alpha) =  \partial (f(\alpha))$.  
\newline
Now suppose $(df_{1})_{\alpha} = (df_{2})_{\alpha}$.  As  $d(f-f({\alpha}))_{\alpha} = df_{\alpha}$ we may assume $f_{1}, f_{2}\in {\mathcal M}_{\alpha}$. 
By Fact 1.1, $f_{1} + {\mathcal M}_{\alpha}^{2} = f_{2} + {\mathcal M}_{\alpha}^{2}$.    
By (*), $\partial '$ preserves both ${\mathcal M}_{\alpha}$ and ${\mathcal M}_{\alpha}^{2}$. 
Hence $\partial '(f_{1} + {\mathcal M}_{\alpha}^{2}) = \partial '(f_{2} + {\mathcal M}_{\alpha}^{2}$. So $d(\partial ' (f_{1}))_{\alpha} = d(\partial '(f_{2}))_{\alpha}$. 
\newline
(ii) By what we have just seen, clearly $\partial '$  equips the $K$-vector space ${\mathcal M}_{\alpha}/{\mathcal M}_{\alpha}^{2}$ with a  $\partial$-module structure.   So we see that via the isomorphism from Fact 1.1 we obtain (ii).

\begin{Corollary} Let again $(X,\partial ') = (X,{\bar s})$ be a rational $D$-variety over the (algebraically closed) differential field $(K,\partial)$.  Then 
\newline
(i) For $f\in K(X)$, defining $D^{1}(df) = d(\partial ' (f))$ gives a well defined additive map from $\{df: f\in K(X)\}\subseteq {\Omega}(X)$ to $\Omega(X)$. 
\newline
(ii) the map in (i) extends uniquely to a map, also called $D^{1}:\Omega(X) \to \Omega(X)$, giving ${\Omega}(X)$ the structure of a $\partial '$-module over $K(X)$. 
\end{Corollary} 
\begin{proof}  (i) We may assume that $(K,\partial)$ is differentially closed. 
Suppose  $f_{1}, f_{2}\in K(X)$ and $df_{1} = df_{2}$. 
We have to show that $d(\partial ' (f_{1})) = d(\partial ' (f_{2}))$, namely for a Zariski-dense set of $\alpha\in X(K)$, $d(\partial ' (f_{1}))_{\alpha} = d(\partial ' (f_{2}))_{\alpha}$.
But by Lemma 2.2, $d(\partial ' (f_{1}))_{\alpha} = d(\partial ' (f_{2}))_{\alpha}$ for all $\sharp$-points $\alpha$ of $X(K)$ (inside some Zariski open set).  So  we are finished by Fact 1.7. 
(Additivity is free as both $d$ and $\partial '$ are additive.)
\newline
(ii)  Actually the formula  $D^{1}(fdg) = \partial '(f)dg + f(D^{1}(g))$ is more or less forced on us by writing $D^{1}(d(fg))$ in two ways and equating them.  That is
$D^{1}(d(fg)) = d(\partial '(fg)) = d((\partial ' f)g) + f{\partial '}(g)) = g d(\partial '(f)) + {\partial '} (f) dg  + fd{\partial '}(g)) + {\partial '}(g)df$
\newline
$ = {\partial '}(g)df + g D^{1}(df)   + {\partial ' }(f) dg +  fD^{1}(dg)$.
Secondly
$D^{1}(d(fg)) = D^{1}(gdf + fdg) = D^{1}(gdf) + D^{1}(fdg)$, which by functoriality yields:
\newline
$D^{1}(gdf) =  {\partial '}(g)df + g D^{1}(df)$ and
\newline
$D^{1}(fdg) =  {\partial '} (f) dg +  fD^{1}(dg)$. 

\end{proof} 

\begin{Remark}  As discussed in the proof of Fact 1.7 we can  embed $(K(X),\partial')$ into $({\mathcal U},\partial)$ over $K$, so if ${\bar a}$ is the image of the coordinate functions ${\bar x}$ then ${\bar a}$ is both a generic point of $X$ over $K$ as well as a $\sharp$-point of $(X,{\bar s})$ in ${\mathcal U}$.  
In any case $\partial '$ on $K(X)$ now identifies with $\partial$ on $E = K({\bar a})$, ${\Omega}_{K}(X)$ identifies with the $E$-vector space $\Omega(E/K)$ of K\"ahler differentials on $E$ over $K$.   Under these identifications $D^{1}$ gives  $\Omega(E/K)$ the structure of a $\partial$-module over $E$.  This coincides with the construction Ax attributes to J. Johnson \cite{Johnson1}, \cite{Johnson2} in \cite{Ax}, and also what Kirby \cite{Kirby}, p. 468, calls the Lie derivative $L_{\partial}$  (except there $K$ is a field of constants of $\partial$). 
\end{Remark}

Finally this all extends to the proalgebraic setting  as discussed in sections 1.1 and 1.2.  So we now have an irreducible proalgebraic variety over $K$, $X_{\infty} = \varprojlim X_{n}$ with field of rational functions $K(X_{\infty}) 
=\cup_{n}K(X_{n})$, and  the $K(X_{\infty})$-vector space  $\Omega_{K}(X_{\infty})$ which is the union of the $K(X_{n})$-vector spaces $\Omega_{K}(X_{n})$. And we have the structure of a rational $D$-variety $(X_{\infty},\partial ')$ on 
$X_{\infty}$ given by an extension of $\partial$ to a derivation $\partial '$ on $K(X_{\infty})$.  Then Corollary 2.3  goes through to equip $\Omega(X_{\infty})$ with the structure $D^{1}$ of a $\partial '$-module over $K(X_{\infty})$.

In any case  we can again embed  $(K(X_{\infty}), \partial ')$ into $({\mathcal U}, \partial)$ over $K$, and as in Remark 2.4, obtain $E = K({\bar a}_{\infty})$, a differential subfield of ${\mathcal U},\partial)$, with a $\partial$-module structure $D^{1}$ on $\Omega(E/K)$ over $E$. 

\section{Ax's theorem} 
We just give a commentary on Ax's proof of the main theorem in \cite{Ax}, what is known as ``Ax-Schanuel for integer powers of the multiplicative group" or ``Ax-Schanuel for the exponential map". 
Anyway what I say here is relatively straightforward at the conceptual level with some simplifications due to Kirby \cite{Kirby}.  We also make use of the following basic duality: given an  irreducible variety $X$ over an algebraically closed field $C$, the $C(X)$-vector space $\Omega_{C}(X)$ of rational differential forms over $C$ on $X$, and the $C(X)$-vector space of rational vector fields on $X$.  Each is the dual space of the other. Algebraically taking $F = C(X)$ it is the duality between the $F$-spaces $\Omega(F/C)$ and $Der(F/C)$ (derivations $F\to F$ which are trivial on $C$).  For example given $D\in Der(F/C)$, $D$ acts $F$-linearly on $\Omega(F/C)$ via ${\hat D}(df) = D(f)$. Again this all extends to proalgebraic varieties, equivalently to the case of $F$ possibly infinitely generated over $C$.

Recall:
\begin{Proposition} (Ax)
Let $(F,\partial)$ be a differential field (characteristic $0$) with algebraically closed field $C$ of constants.
Let $a_{1},..,a_{n}, b_{1},..,b_{n} \in F$ such that $a_{1},..,a_{n}$ are $\Q$-linearly independent modulo $C$, the $b_{i}$ are nonzero, and 
$\partial a_{i} = \partial b_{i}/b_{i}$ for $i=1,..,n$. Then $tr.deg(C(a_{1},..,a_{n}, b_{1},..,b_{n}/C) \geq n+1$.
\end{Proposition}
\begin{proof}  Let $E = C({\bar a}, {\bar b})$ the function field of the irreducible variety $X$ over $C$ with generic (over $C$) point $({\bar a}, {\bar b})$ (where ${\bar a} = (a_{1},..,a_{n})$ and ${\bar b} = (b_{1},..,b_{n})$). 
Let $({\bar a}, {\bar b})_{\infty} = \nabla_{\infty}({\bar a}, {\bar b}) = (((\partial^{n}({\bar a}),\partial^{n}({\bar b})): n=0,1,....)$, 
a generic over $C$ point of an irreducible proalgebraic variety $X_{\infty}$ over $C$ projecting dominantly to $X$.  We may assume that $F = C(({\bar a},{\bar b})_{\infty})$
So we have a rational (proalgebraic) $D$-variety structure on $X_{\infty}$ given by $(F,\partial)$. 
From Section 3 we have a $\partial$-module over $F$, $\Omega_{C}(X_{\infty}) = \Omega(F/C)$ with the derivation $D^{1}$. 
$\Omega_{C}(X) = \Omega(E/C)$ is a $E$-vector subspace of $\Omega(F/C)$. 

Let us assume for a contradiction that $tr.deg(E/C) = dim(X)$ is $\leq n$.  By  Fact 1.2, $dim(\Omega(E/C))$ as an $E$-vector space is $\leq n$. Let $V$ be the $F$-vector subspace of $\Omega(F/C)$ obtained from $\Omega(E/C)$ by 
tensoring with $F$.  So the $F$-dimension of $V$ is $\leq n$. By the remarks at the beginning of the section the derivation $\partial$ on $F$ gives rise to a linear form $\hat\partial$ from $\Omega(F/C)$ to $F$ whose restriction $
\hat\partial$ to $V$ is also a linear form on $V$.
Let $\omega_{i} = da_{i} - db_{i}/b_{i} \in \Omega(E/C)$ for $i=1,..,n$.   The first claim is part of  Kirby's simplification of Ax's proof. 
\newline
{\bf Claim 1.}  (i) ${\hat\partial}|V$ is onto $F$.
\newline
(ii)  $Ker({\hat\partial}|V)$ has $F$-dimension $< n$. 
\newline
(iii) Each $\omega_{i}\in Ker({\hat\partial}|V)$
\newline
{\em Proof}  (i) By assumption $a_{1}\notin C$, whereby ${\hat\partial}(da_{1}) = \partial (a_{1}) \neq 0$, hence the image under $\hat\partial$ of the $F$-linear span of $da_{1}$ is all of $F$. 
\newline
(ii) follows from (i) as the $F$-dimension of $Ker({\hat\partial}|V)$ is $dim(V) - 1$.
\newline
(iii)  ${\hat\partial}(da_{i} - (1/b_{i})db_{i}) = {\hat\partial}(da_{i}) - (1/b_{i}){\hat\partial}(b_{i}) = \partial(a_{i} ) - (1/b_{i})\partial(b_{i}) = 0$ by the assumption of the proposition.  \qed

\vspace{5mm}
\noindent
It follows from (ii) and (iii) of Claim 1, that
\newline
{\bf Claim 2.}  $\omega_{1},...,\omega_{n}$ are $F$-linearly dependent.

\vspace{5mm}
\noindent
{\bf Claim 3.}  (see Lemma 3 of \cite{Ax}.) Each $\omega_{i}\in \Omega(F/C)$ is a solution of $D^{1} = 0$ (where recall $D^{1}$ gives the $\partial$-module structure on $\Omega(F/C)$ over $F$ induced by $\partial$ on $F$ as detailed in the previous section). 
\newline
{\em Proof.} Fix $i = 1,..,n$ and write $a$ for $a_{i}$ and $b$ for $b_{i}$. Then using Corollary 2.3 (with $X_{\infty}$ in place of $X$ and so $F = C(X_{\infty})$), we see that
$D^{1}(\omega_{i}) = D^{1}(da - (1/b)db) = D^{1}(da) -D^{1}((1/b)db) =  d(\partial (a)) - (1/b)(d(\partial (b)) +  (\partial (b)/b^2)db$.
\newline
But note that $d({\partial }b/b) = (1/b)d(\partial b) - (\partial b/b^{2})db$, so  substituting $d({\partial }b/b) + (\partial b/b^{2})db$ for $(1/b)d(\partial b)$, we see that 
\newline
$D^{1}(\omega_{i}) = d(\partial a)) - d(\partial b/b) - (\partial b/b^{2})db + (\partial (b)/b^2)db = d(\partial a) - d(\partial b/b) = d(\partial(a) - \partial(b)/b)) = 0$.  

\vspace{5mm}
\noindent
From Claims 2, 3 and Fact 2.1 we conclude:
\newline
{\bf Claim 4.} $\omega_{1}, ..., \omega_{n}$ are $C$-linearly dependent. 

\vspace{5mm}
\noindent
Claim 4 is more or less the main point, from which gets the contradiction in various ways. Ax's  arguments in \cite{Ax} are specific  to the case at hand (the logarithmic derivative on powers of the multiplicative group rather than other semiabelian varieties). 
For now we follow \cite{Ax}, after which we mention another account  which works in the more general environment. 

From Claim 4 we obtain $c_{i}\in C$ for $i=1,..,n$ not all $0$ such that  $\sum_{i}c_{i}\omega_{i} = 0$  (in $\Omega(E/C)$).  So $\sum_{i} c_{i}(db_{i}/b_{i}) = \sum_{i} c_{i}da_{i}$.
Then  
\newline
(*) $\sum_{i}c_{i}(db_{i}/b_{i}) = d(\nu)$
\newline
 where $\nu = \sum_{i}c_{i}a_{i}$. 

\vspace{5mm}
\noindent
{\bf Claim 5.}  The $db_{i}/b_{i}$ are ${\mathbb Z}$-linearly dependent. 

\vspace{2mm}
\noindent
Note that Claim 5 gives a contradiction to the assumptions, because applying $\hat\partial$ to a suitable ${\mathbb Z}$-linear combination of the $db_{i}/b_{i}$ will yield that some nontrivial $\mathbb Z$-linear combination of  the $\partial b_{i}/b_{i}$ is $0$ and so some nontrivial ${\mathbb Z}$-linear combination of the $\partial a_{i}$ is $0$, yielding that the $a_{i}$ are $\Q$-linearly dependent, modulo $C$, contradicting our assumptions. 
So it suffices to prove Claim 5.

\vspace{2mm}
\noindent
{\em Proof of Claim 5.}   Assume that the claim fails. Then by minimizing the number of $c_{i}$ needed we may assume that the $c_{i}$ are $\Q$-linearly independent. 
Ax first assumes $E$ to be an ``algebraic function field", namely finitely generated and of transcendence degree $1$ over $C$ and appeals to results on orders and residues at places $p$ of $F$ over $C$. 
For each such $p$ there is a valuation $ord_{p}: F \to \Z \cup\infty$ and $C$-linear map $res_{p}: \Omega(E/C) \to C$.  For any nonzero  $e\in E$,  $res_{p}(de/e) = ord_{p} (e)$ and for any $e\in E$, $res_{p}e = 0$.
So  comparing the residues at a place $p$ of the left and right hand side of (*), we see that $0 = res_{p}(\sum_{i}c_{i}(db_{i}/b_{i})) = \sum_{c_{i}}ord_{p}b_{i}$.  As the $c_{i}$ are assumed to be $\Q$-linearly independent, it follows that for 
each $p$, $v_{p}(b_{i}) = 0$ for all $i$.  This implies that each $b_{i}$ is a constant function and we get a contradiction (for each $i$, $\partial b_{i}/b_{i} = 0$ so $\partial a_{i} = 0$, so $a_{i}$ is a constant.)

Ax then reduces to this case, by observing  that if $E_{1}$ is a relatively algebraically closed subfield of $E$ which contains $C$ and with $tr.deg(E/E_{1}) = 1$ we can apply the canonical epimorphism from $\Omega(E/C)$ to $\Omega(E/E_{1})$ and the above argument to deduce that $b_{i}\in E_{1}$. But the intersection of all such $E_{1}$'s is $C$. Hence again $b_{i}\in C$ for all $i$, a contradiction. \qed

So this ends the proof of Proposition 3.1 following \cite{Ax}
 
\end{proof}

\vspace{5mm}
\noindent
We now discuss briefly  another approach to the conclusion of the proof of Proposition 3.1 which appears in \cite{Kirby}  the proof of which was suggested by Kowalski. This has the advantage of working in the more general context of a semiabelian variety in place of an algebraic torus. 
So we are in the same context as Proposition 3.1 but we assume we have got to Claim 4 that the $\omega_{i}$ are $C$-linearly dependent, say $\sum_{i=1,..,n} c_{i}\omega_{i} =0$ for $c_{i}\in C$, not all $0$. The argument now is purely algebro-geometric (as was the last part of the argument in the earlier proof).
Let $G$ be the algebraic group ${\mathbb G}_{a}^{n}\times {\mathbb G}_{m}^{n}$ which we will write as $H\times S$. 
Let now $K$ be the algebraic closure of the field $E$. 
The affine variety $X$ over $C$ of which ${(\bar a}, {\bar b})$ was a $C$-generic point is a subvariety of $G$.
We now consider the function field  $K(G)$ of $G$ over $K$.
We now consider the differential forms $\omega_{i}$ above as rational differential forms over $K$ on $G$, so elements of $\Omega_{K}(G)$ or equivalently of $\Omega(K(G)/K)$. So $\omega_{i} = dx_{i} - dy_{i}/y_{i}$.  These rational differential forms on $G$ are $K$-linearly independent and invariant under left multiplication by elements of $G(K)$. Hence $\omega = \sum_{i}c_{i}\omega_{i}$ is nonzero and $G(F)$-invariant.

Let $\alpha$ denote the $K$-point $({\bar a}, {\bar b})$ of $G$.  With the current notation, Claim 4 translates into  $\omega_{\alpha} = 0$.  By a result attributed by Kolchin \cite{Kolchin} to Rosenlicht \cite{Rosenlicht},
$\{\beta\in G(F): \omega_{\beta} = 0\}$ is a subgroup  of $G(F)$ which is moreover proper as $\omega\in \Omega_{K}(G)$ is nonzero. 

Now consider the irreducible subvariety $X$ of $G$. $X$ is defined over $C$ with $C$-generic point $\alpha$. Note that $\omega_{\gamma} = 0$ for all $\gamma\in G(C)$. So  translating $X$ by a suitable point of $G(C)$, we obtain an irreducible subvariety of $G$ defined over $C$ and containing the identity of $G$ and whose generic over $C$ point $\beta$ say also satisfies $\omega_{\beta} = 0$.  Let $H$ be the subgroup of $G$ generated by $Y$.  Then $H$ is connected, defined over $C$  and its generic over $C$ point $\nu$ is a sum of $C$-generic points of $Y$, hence (as $\omega$ is defined over $C$), $\omega_{\nu} = 0$. So $H$ is contained in $S$ hence is a proper algebraic subgroup of $G$.  Now $H$ is a product of algebraic subgroups (over $C$), $H_{1}$ of ${\mathbb G}_{a}^{n}$ and $H_{2}$ of ${\mathbb G}_{m}^{n}$.  Finally check that $H_{2}$ is a proper connected algebraic subgroup of ${\mathbb G}_{m}^{n}$,  so defined by a finite set of equations $z_{1}^{d_{1}}....z_{n}^{d_{n} } = 1$ with the $d_{i}\in \Z$,  not all $0$. It follows from the construction that the $a_{i}$ are $\Q$-linearly dependent modulo $\C$, a contradiction. 

\vspace{5mm}
\noindent
Another simple account of the last step  (Claim 5) appears  in \cite{Marker}  using instead ideas of Rosenlicht, as mentioned earlier.

\end{document}